\documentclass[twoside, 12pt]{amsart}
\usepackage{latexsym}
\usepackage{amssymb,amsmath,amsopn}
\usepackage{color,epsfig}      %To insert ps figures

\newtheorem{thm}{Theorem}
\newtheorem{lem}{Lemma}

\newtheorem{defn}{Definition}
\newtheorem{rem}{Remark}

\newtheorem{ques}{Question}

\DeclareMathOperator{\conv}{conv}
\DeclareMathOperator{\Sym}{Sym}

\DeclareMathOperator{\dif}{d}

\DeclareMathOperator{\area}{area}

\DeclareMathOperator{\bd}{bd}

\DeclareMathOperator{\diam}{diam}
\DeclareMathOperator{\SIM}{SIM}
\DeclareMathOperator{\AFF}{AFF}

\newcommand{\K}{\mathcal{K}}

\renewcommand{\Re}{\mathbb R}
\newcommand{\Sph}{\mathbb{S}}

\newcommand{\M}{\mathcal{M}}

\begin{document}
\title[A shape evolution model]{A shape evolution model under affine transformations}

\author[G. Domokos, Z. L\'angi and M. Mezei]{G\'abor Domokos, Zsolt L\'angi and M\'ark Mezei}
\address{G\'abor Domokos, MTA-BME Morphodynamics Research Group and Dept. of Mechanics, Materials and Structures, Budapest University of Technology,
M\H uegyetem rakpart 1-3., Budapest, Hungary, 1111}
\email{domokos@iit.bme.hu\\}
\address{Zsolt L\'angi, MTA-BME Morphodynamics Research Group and Dept.\ of Geometry, Budapest University of
Technology and Economics, Budapest, Egry J\'ozsef u. 1., Hungary, 1111}
\email{zlangi@math.bme.hu\\}
\address{M\'ark Mezei, Dept. of Mechanics, Materials and Structures, Budapest University of Technology,
M\H uegyetem rakpart 1-3., Budapest, Hungary, 1111}
\email{mezmarc@gmail.com}


\begin{abstract}
In this note we describe a discrete dynamical system acting on the similarity classes of a plane convex body within the affine class of the body. We find invariant elements in all affine classes, and describe the orbits of bodies in some special classes. We point out applications with abrasion processes of pebble shapes.
\end{abstract}

\maketitle

\section{Introduction}\label{sec:intro}

In convex geometric research, it is a flourishing area to examine affinely invariant problems: problems in which affine copies of a convex body are not distinguished \cite{L98, Sch93}. Similarly, there are numerous results which identify only congruent copies of a convex body \cite{BMP05, G07}.
In this paper we present a problem dealing with convex bodies only within a given affine class, and identify all similar copies of the body.

Our study is also motivated by the recent, rising interest in shape evolution models in geomorphology.
Physical abrasion processes, composed of collisional and frictional abrasion, are of fundamental importance in the evolution of sedimentary particles.
Geologists try to track the shape evolution process by measuring scalar quantities called shape descriptors associated with the particle's global shape.
The most common shape descriptors are the axis ratios of the approximating ellipsoid \cite{Zingg} and roundness \cite{Cox}  of the orthogonal projection (sometimes
also referred to as circularity \cite{Blott}).
While substantial amount of data on shape descriptors has accumulated over decades (e.g. \cite{Bluck, Carr, Zingg}), many aspects of the evolution of these shape descriptors in
mathematical abrasion models are still not fully understood. In the mathematical theory of collisional abrasion \cite{Bloore, Firey} it has always been assumed
that the \emph{driving force} of the evolution process is curvature, i.e. abrasion is determined not by global geometric features, rather, by local quantities. Mathematical models of curvature-driven abrasion can be either continuous (partial differential equations) where mass is removed in continuous time \cite{Bloore, Firey} or discrete (so-called chipping models) where finite amount of mass is removed in discrete time-steps \cite{DSV, Domokosetal2, Krapivsky}. Chipping models are, in general, not the rigorous discretizations of curvature-driven partial differential equations. Frictional abrasion, on the other hand, appears to be of entirely different
nature. In \cite{DG1}  a set of axioms for frictional abrasion is proposed and these axioms admit geometric theories where abrasion is driven not by local quantities, rather,
by global shape descriptors. Not only is the experimental verification of such models still lacking,
their mathematical theory has not been explored either. Here we take a first step and investigate an axis-ratio-driven model which operates in discrete time, so it could be regarded as a distant relative of chipping models of curvature-driven abrasion. Our goal is \emph{not} to build a realistic model of 3D physical abrasion, rather, to study
some key mathematical elements of such models (such as discrete event sequence, affinity-driven shape evolution) in their own right in a rigorous fashion. Nevertheless, we hope that this study may help to get qualitative insight into some physical aspects of abrasion processes.

We start with some preliminary concepts.

\begin{defn}\label{defn:width}
Let $K$ be a plane convex body. For a direction $u \in \Sph^1$, $w_K(u)$ and $l_K(u)$ denotes the \emph{width} of $K,$ and the length of a longest chord of $K$ in this direction, and let $u^\perp$ denote the direction perpendicular to $u$.
\end{defn}

By continuity, it is easy to show that for any plane convex body $K$ there is a square circumscribed about $K$. This motivates us to introduce the following.

\begin{defn}\label{defn:extremal}
Let $K$ be a plane convex body. For any rectangle $R$ circumscribed about $K$, let $a(R) \leq b(R)$ denote the lengths of the sides of $R$.
We say that $R$ is \emph{extremal} among the rectangles circumscribed about $K$ (or simply that $R$ is an \emph{extremal} rectangle of $K$), if
\[
\frac{a(R)}{b(R)} = \inf \left\{ \frac{a(R')}{b(R')} : R' \hbox{ is a rectangle circumscribed about } K \right\}.
\]
\end{defn}

Our main definition is the following.

\begin{defn}\label{defn:affinity}
Let $K$ be a plane convex body, and let $R$ be an extremal rectangle circumscribed about $K$, with side lengths $a(R) \leq b(R)$.
Let $L$ be a line through the origin $o$, such that $L$ is parallel to the sides of $R$ with length $a(R)$. Let $f_{K,R}$ denote the orthogonal affinity with axis $L$ and ratio $\frac{a(R)}{b(R)}$.
\end{defn}

Note that, according to our definition, $f_{K,R}(R)$ is a square, and $f_{K,R}$ is well-defined, since if $R$ is a square (i.e. the direction of the sides with length $a(R)$ is not well-defined), then $f_{K,R}$ is the identity.
On the other hand, since a plane convex body may have more than one extremal rectangle, $K$ alone may not determine $f_{K,R}$.
If it does; that is, if $K$ has a unique extremal rectangle, we set $f_K=f_{K,R}$.

In the paper we use the following notations. We denote the family of plane convex bodies by $\K$, and the family of $o$-symmetric plane convex bodies by $\M$.
For any $K \in \K$, we set
\begin{equation}
\begin{array}{rcl}
\AFF( K ) & = & \{ K' \in \K : K' \hbox{  is an affine copy of } K \} ;\\
\SIM(K) & = & \{ K' \in \K : K' \hbox{  is similar to } K \} ;\\
F(K) & = & \left\{ \SIM \left( f_{K,R}(K) \right) : R \hbox{ is an extremal rectangle of } K \right\}.
\end{array}
\end{equation}

Note that $F(K)$ may contain more than one similarity class, if the extremal rectangle of $K$ is not unique.  Furthermore, we observe
that for any plane convex body $K$, every member of $F(K)$ is contained in $\AFF(K)$.

\begin{defn}\label{invariance}
Let $K$ be a plane convex body. We say that $K$ is
\begin{itemize}
\item \emph{weakly invariant}, if some similarity class in $F(K)$ is $\SIM(K)$;
\item \emph{invariant}, if the only element of $F(K)$ is $\SIM(K)$;
\item \emph{strongly invariant}, if every rectangle circumscribed about $K$ is a square.
\end{itemize}
\end{defn}

Note that strong invariance implies invariance, which implies weak invariance.

Our main goal is the investigation of the orbits of a plane convex body $K$ under subsequent applications of $f_{K,R}$. The points in the phase space of this iteration are similarity classes associated with the elements of the affine class of $K$.
Since each triangle belongs to the same affine class, all orbits starting with triangle can be represented in the same phase space.
One of our main results is Theorem  \ref{thm:triangles1}, stating that there is one global attractor in this phase space and we also identify this attractor.
 Moreover, we characterize strongly invariant plane convex bodies (Theorem \ref{thm:characterization}), and for centrally symmetric polygons with a few number of vertices,
we find their weakly invariant affine copies (Theorem \ref{thm:hexagons}).

The paper is structured as follows:  in Section~\ref{sec:stronginvariance} we collect the properties of strongly invariant plane convex bodies.
In Section~\ref{sec:triangles} we determine the orbit of any triangle under subsequent applications of the affine transformations defined in Definition~\ref{defn:affinity}, and carry out the same for a wider class of transformations as well.
In Section~\ref{sec:hexagons} we characterize the weakly invariant $o$-symmetric convex hexagons.
Finally, in Section~\ref{sec:remarks} we collect our additional remarks and ask some open questions.

\section{Strongly invariant plane convex bodies}\label{sec:stronginvariance}

Our main result in this section is the following.

\begin{thm}\label{thm:characterization}
Let $K \in \K$. Then
\begin{itemize}
\item[(\ref{thm:characterization}.1)] $K$ is strongly invariant if, and only if its central symmetral $\frac{1}{2} (K-K)$ has a $4$-fold rotational symmetry.
\item[(\ref{thm:characterization}.2)] Up to similarity, the affine class $\AFF(K)$ of $K$ contains at most one strongly invariant element.
\end{itemize}
\end{thm}

\begin{proof}[Proof of (1.1)]
If $R$ is a rectangle circumscribed about $K$, with sides parallel to $u, u^\perp \in \Sph^1$, then the sides of $R$ are of lengths $w_K(u)$ and $w_K(u^\perp)$.
Since for any $K$ and any $u \in \Sph^1$, we have $w_K(u)= w_{\frac{1}{2}(K-K)}(u)$, it follows that $K$ is strongly invariant if, and only if $\frac{1}{2}(K-K)$ is.

Now, assume that $K$ is $o$-symmetric. Then for any $u \in \Sph^1$, we have $w_K(u) = 2 h_K(u)$, where $h_K$ is the \emph{support function} of $K$.
Thus, $K$ is strongly invariant if, and only if for any $u \in \Sph^1$, we have $h_K(u) = h_K(u^\perp)$, or in other words, if $h_K$ has a $4$-fold rotational symmetry. This is clearly equivalent to $K$ having a $4$-fold rotational symmetry, and the assertion follows.
\end{proof}

To prove (\ref{thm:characterization}.2), we need a lemma.

\begin{lem}\label{lem:circumscribed}
Let $K$ be an $o$-symmetric plane convex body having a $4$-fold rotational symmetry, and let $P$ be a parallelogram circumscribed about $K$. Then the following are equivalent.
\begin{itemize}
\item[(\ref{lem:circumscribed}.1)] $P$ has minimal area over the family of parallelograms circumscribed about $P$.
\item[(\ref{lem:circumscribed}.2)] $P$ is a square, and the midpoints of its sides belong to both $K$ and its incircle.
\end{itemize}
\end{lem}

\begin{proof}[Proof of Lemma~\ref{lem:circumscribed}]
Let $C$ be the incircle of $K$. Then, by the symmetry of $K$, the incentre of $K$ is $o$.

Consider some point $p \in C \cap \bd K$. As $C \subseteq K$, the tangent line of $C$ at $p$ supports $K$.
Let $p_2, p_3, p_4$ be the rotated copies of $p=p_1$, around $o$, by angles $\frac{\pi}{2}$, $\pi$, $\frac{3\pi}{2}$, respectively.
Let $S_p$ be the square with the tangent lines of $C$ at $p_1$, $p_2$, $p_3$ and $p_4$ as its sidelines.
Clearly, $S_p$ is circumscribed about $K$, and the midpoints of its sides are common points of $\bd K$ and $C$.
Finally, the area of $S$ is $4 \rho^2$, where $\rho$ is the inradius of $K$.
Since $S_p$ is a parallelogram circumscribed about $K$, we obtain that if $P$ is a minimum area parallelogram circumscribed about $K$,
then $\area(P) \leq 4 \rho^2$.

On the other hand, for any minimum area parallelogram $P$ we have $\area(P) = l_K(u) w_K(u^\perp)$ for some $u \in \Sph^1$.
Thus, the inequalities $2\rho \leq l_K(v) \leq w_K(v)$ for any $v \in \Sph^1$ readily imply that $\area(P) = 4 \rho^2$, and that all sides of $P$ touch $C$. Hence, (\ref{lem:circumscribed}.1) yields (\ref{lem:circumscribed}.2).
%Now, consider some minimum area parallelogram $P$. Note that if $P$ has a side parallel to some , then
%$\area(P) = l_K(u) w_K(u^\perp)$ (cf. Definition~\ref{defn:width}).
%Furthermore, observe that $l_K(u) \geq 2\rho$, and $w_K(u^\perp) \geq l_K(u^\perp) \geq 2\rho$, which, together with the previous paragraph, yields that $\area(P) = 4\rho^2$. Thus, (\ref{lem:circumscribed}.2) implies (\ref{lem:circumscribed}.1).
%Furthermore, from $\area(P) = 4 \rho^2$ it follows that  $l_K(u) = l_K(u^\perp) = w_K(u^\perp) = 2\rho$. This implies that the sides of $P$ parallel to $u$ touch $C$, and are of length $2\rho$. Since this property holds for any side of $P$, $P$ is an equilateral parallelogram, with right angles, circumscribed about $C$; that is, a square. Hence, (\ref{lem:circumscribed}.1) yields (\ref{lem:circumscribed}.2).
\end{proof}

Now we turn to the proof of (\ref{thm:characterization}.2) of Theorem~\ref{thm:characterization}.

\begin{proof}[Proof of (1.2)]
Note that if $g$ is any affine transformation, then for any $K \in \K$, we have $\frac{1}{2} \left(g(K)-g(K) \right) = g\left( \frac{1}{2}(K-K) \right)$, which yields that for any $K' \in \AFF( K )$, it follows that $\AFF( \frac{1}{2}(K'-K') ) = \AFF( \frac{1}{2}(K-K) )$.
Recall the observation from the proof of (\ref{thm:characterization}.1) that for any $K \in \K$, $K$ is strongly invariant if, and only if $\frac{1}{2}(K-K)$ is strongly invariant. This means that, up to similarity, for any  $K \in \K$ the number of strongly invariant elements of $\AFF( K )$ is equal to the number of strongly invariant elements in $\AFF( \frac{1}{2}(K-K) )$.
Thus, it suffices to prove (\ref{thm:characterization}.2) under the assumption that $K$ is $o$-symmetric.

Assume that $K$ has a $4$-fold rotational symmetry, and its centre is the origin. Let $g$ be any area-preserving linear transformation, and assume that $g(K)$ has a $4$-fold rotational symmetry. We show that in this case $g(K)$ and $K$ are congruent, which clearly implies the assertion.

Observe that $P$ is a minimum area parallelogram circumscribed about $K$ if, and only if $g(P)$ is a minimum area parallelogram circumscribed about $g(K)$.
Consider some such parallelogram $P$. Then, by Lemma~\ref{lem:circumscribed}, $P$ and $g(P)$ are squares circumscribed about the incircles of $K$ and $g(K)$, respectively. Furthermore, since $P$ and $g(P)$ are of equal size, the incircles of $K$ and $g(K)$ coincide, which implies that $g(P)$ is a rotated copy of $P$.
Since the effect of $g$, say, on the vertices of $P$, determines $g$, it follows that $g$ is either a rotation, or the composition of a rotation and a reflection about a line. In both cases, $g$ is a congruence, and thus, $K$ and $g(K)$ are congruent.
\end{proof}

\section{The orbits of triangles}\label{sec:triangles}

Before stating our main result in this section, we set $x_0 = \frac{2}{3} -\frac{\sqrt[3]{44+12 \sqrt{69}}}{6} + \frac{10}{ 3 \sqrt[3]{44+12\sqrt{69}}}$, and let the triangle $T_0$ be defined as the one with vertices $(0,0)$, $(1,0)$ and $(x_0,1)$.

\begin{thm}\label{thm:triangles1}
Let $T$ be a triangle.
\begin{enumerate}
\item[(\ref{thm:triangles1}.1)] $T$ is weakly invariant if, and only if $T$ is similar to $T_0$. Furthermore, in this case $T$ is invariant, but not strongly invariant.
\item[(\ref{thm:triangles1}.2)] If $f_{T,R}(T)$ is not similar to $T_0$ for any extremal rectangle $R$ of $T$, then any sequence $\{T_k \}$, $k=1,2,3,\ldots$, with the property that $T_{k+1}=\frac{1}{\diam f_{T_k,R_k}(T_k)} f_{T_k,R_k}(T_k)$ for some extremal rectangle of $T_k$, contains mutually non-congruent elements, and, according to the topology defined by Haussdorff distance up to congruence, converges to $T_0$.
\end{enumerate}
\end{thm}

Before proving Theorem~\ref{thm:triangles1}, we prove Lemma~\ref{lem:trianglesextremal}.

\begin{lem}\label{lem:trianglesextremal}
Let $T$ be a triangle with sides of length $x \leq y \leq z$. Then the boundary of any extremal rectangle of $T$ contains a side of length $z$.
\end{lem}

\begin{proof}
Let the lengths of the altitudes of $T$, perpendicular to the sides of length $x$, $y$ and $z$, be denoted by $h_x$, $h_y$ and $h_z$, respectively.
 Note that for any direction $u \in \Sph^1$, we have that $z \cdot h_z = 2 \area(T) = l_u (T) w_{u^{\perp}}(T) \leq w_u (T) w_{u^{\perp}}(T)$. Since $z = \diam (T) = \max \{ w_{u}(T) : u \in \Sph^1 \}$ and $h_z$ is the width of $T$ in the perpendicular direction, we have that
$h_z \leq w_u(T) \leq z$ for any $u \in \Sph^1$. Here, we have equality if, and only if $u$ is parallel to an altitude of length $h_z$ or a side of length $z$, respectively. Since for any rectangle $R$, circumscribed about $T$ and having a side parallel to $u \in \Sph^1$, the side-lengths of $R$ are $w_u(T)$ and $w_{u^\perp}(T)$, it follows that the boundary of any extremal rectangle contains a side of length $z$.
\end{proof}

Observe that by Lemma~\ref{lem:trianglesextremal}, unless $T$ is isosceles, there is a unique extremal rectangle of $T$, and if $T$ is isosceles, then all the extremal rectangles of $T$ are congruent.

\begin{proof}[Proof of Theorem~\ref{thm:triangles1}]
Let $T$ be any triangle. Note that it suffices to prove the assertion under the assumptions that $T$ has a side of unit length, and that the altitude of $T$, perpendicular to this side, is of unit length as well. Thus, we may assume that the vertices of $T$ are $a_1=(0,0)$, $a_2=(1,0)$ and $c=(x,1)$ in a suitable Cartesian coordinate system. Furthermore, without loss of generality, we may assume that $0 < x \leq \frac{1}{2}$, implying that $[a_2,a_3]$ is a longest side of $T$. Note that up to congruence, $T$ is determined by the value of $x$, which we call the \emph{parameter} of $T$.

An elementary computation yields that $|a_3-a_2| = \sqrt{1+(1-x)^2}$, and that the orthogonal projection $q$ of $a_1$ on this side is at the distance
$|a_2-q| = \frac{1-x}{\sqrt{1+(1-x)^2}}$ from $a_2$. Note that as $q$ is a point of the Thales circle of the segment $[a_1,a_2]$, we have $|a_2 - q | < |a_3-q|$.

First, we show (\ref{thm:triangles1}.1). To do this, we need to characterize the values of $x$ with the property that $f(x)=\frac{|a_2-q|}{ |a_3-a_2|}=\frac{1-x}{1+(1-x)^2}$ is equal to $x$. Thus, we have an equation for $x$, which can be transformed into the form $x^3-2x^2+3x-1=0$. This equation has only one real root, namely $x_0$.

Now consider the sequence $\{ T_k \}$. Note the if the parameter of $T_k$ is $x_k$, then $T_{k+1}$ has a congruent copy with parameter $x_{k+1}=\frac{1-x_k}{1+(2-x_k)^2}$. Let us examine the expression $f(x)$ defining the operation. An elementary computation shows that $f''(x)$ is negative on the interval $\left[ 0, \frac{1}{2} \right]$, which yields that on this interval $ - \frac{12}{25} \leq f'(x) = \frac{-2x+x^2}{\left( 1+(1-x)^2 \right)^2} \leq 0$.
Hence, under the iteration the sequence $\{ x_{k} \}$ converges to $x_0$, with the property that $\left| x_{k+1} - x_ 0\right| \leq \frac{12}{25} \left| x_k - x_0 \right|$.
\end{proof}

We can generalize Definition~\ref{defn:affinity}.

\begin{defn}\label{defn:affinity2}
Let $K$ be a plane convex body, $R$ be an extremal rectangle of $K$ with side lengths $a(R) \leq b(R)$, and let $\lambda > 0$ be a real parameter. Let $L$ be a line through the origin $o$, such that $L$ is parallel to the sides of $R$ with length $a(R)$. Let $f^\lambda_{K,R}$ denote the orthogonal affinity with axis $L$ and ratio $\frac{\lambda a(R)}{b(R)}$.
\end{defn}

By its definition, we have that $f^\lambda_{K,R}(R)$ is a rectangle of side lengths $a(R)$ and $\lambda a(R)$.
Since for any triangle $T$ all extremal rectangles are congruent, in this case, for brevity, we may use the notation $f^\lambda_T$, and if $f^\lambda_T(T)$ is similar to $T$, then we say that $T$ is \emph{$\lambda$-invariant}.

For simplicity, we formulate our next theorem only for triangles with vertices $(0,0)$, $(1,0)$ and $(x,\lambda)$.
Note that for any triangle $T$, $f^\lambda_T(T)$ is similar to some such triangle.

%To formulate our next theorem, we introduce the following definition.

%\begin{defn}\label{defn:parameter}
%Let $T = \conv \{ a_1, a_2, a_3 \}$ be a triangle such that $[a_1,a_2]$ is a diameter of $T$. Let $q$ denote the orthogonal projection of $a_3$ onto $[a_1,a_2]$.
%Then the quantity $\min \left\{ \frac{|a_1-q|}{|a_1-a_2|},  \frac{|a_2-q|}{|a_1-a_2|} \right\}$ is called the \emph{parameter} of $T$. Furthermore, we define
%the \emph{relative height} of $T$ as $\frac{|a_3-q|}{|a_1-a_2|}$.
%\end{defn}

%Note that the parameter of any triangle as at most $\frac{1}{2}$.

Set
\[
x^\lambda = \frac{\sqrt[3]{28-72\lambda^2+12 \sqrt{9-12\lambda^2+60\lambda^4+12\lambda^6}}}{6}-
\]
\[
-\frac{ \frac{4}{3}+2 \lambda^2}{ \sqrt[3]{28-72\lambda^2+12\sqrt{9-12\lambda^2+60\lambda^4+12\lambda^6}}}+\frac{2}{3},
\]
where $\lambda \geq \frac{\sqrt{3}}{2}$, and let $T^\lambda$ be the triangle with vertices $(0,0)$, $(1,0)$ and $(x^\lambda,\lambda)$.
We note that $x^{\frac{\sqrt{3}}{2}} = \frac{1}{2}$, and that  $T^{\frac{\sqrt{3}}{2}}$ is a regular triangle.

\begin{thm}\label{thm:triangles2}
Let $\{T_k\}$ be a sequence of triangles such that the vertices of $T_k$ are $(0,0)$, $(1,0)$ and $(x_k,\lambda)$, where $0 < x_k \leq \frac{1}{2}$, and $f(T_k)$ is similar to $T_{k+1}$. Set $x=x_1$. Then we have the following.
\begin{enumerate}
\item[(\ref{thm:triangles2}.1)] If $\lambda \geq \frac{\sqrt{3}}{2}$ and $x = x^\lambda$, or if $0 < \lambda \leq \frac{\sqrt{3}}{2}$ and $1-\sqrt{1-\lambda^2} \geq x \leq \frac{1}{2}$, then all the elements of $\{ T_k \}$ are equal.
\item[(\ref{thm:triangles2}.2)] If $\lambda \geq \frac{\sqrt{3}}{2}$ and $x \neq x^\lambda$, then $\{ x_k \}$ have mutually different elements, and $x_k \to x^\lambda$.
\item[(\ref{thm:triangles2}.3)] If $\frac{1}{\sqrt{2}} < \lambda <\sqrt{\frac{\sqrt{5}-1}{2}}$ and $1-\frac{\lambda^2}{\sqrt{1-\lambda^2}} < x$, or if
$\sqrt{\frac{\sqrt{5}-1}{2}} < \lambda < \frac{\sqrt{3}}{2}$ and $0 < x < 1-\sqrt{1-\lambda^2}$, then $\{ x_k \}$ have mutually different elements, and $x_k \to 1-\sqrt{1-\lambda^2}$.
\item[(\ref{thm:triangles2}.4)] If $0 < \lambda < \frac{1}{\sqrt{2}}$ and $0 < x < 1-\sqrt{1-\lambda^2}$, or if $\frac{1}{\sqrt{2}} < < \lambda \sqrt{\frac{\sqrt{5}-1}{2}}$ and $0 < x < 1-\frac{\lambda^2}{\sqrt{1-\lambda^2}}$, then, apart from finitely many, all the elements of the sequence $\{ T_k \}$ are equal.
\end{enumerate}
\end{thm}

Since the proof of Theorem~\ref{thm:triangles2} is a more elaborate version of that of Theorem~\ref{thm:triangles1}, we omit it.

\section{Weakly invariant $o$-symmetric convex hexagons}\label{sec:hexagons}

The aim of this section is to characterize the weakly invariant $o$-symmetric polygons with a small number of vertices.
In our investigation, we regard a parallelogram as a degenerate hexagon.

Note that by Theorem~\ref{thm:characterization}, up to similarity, among the $o$-symmetric (possibly degenerate) hexagons,
the only strongly invariant ones are the squares. The results in Section~\ref{sec:triangles} imply that, again up to similarity, in the class of affinely regular hexagons, there is exactly one weakly invariant hexagon, namely $\frac{1}{2}(T_0-T_0)$ (for the definition of $T_0$, see the first paragraph of Section~\ref{sec:triangles}).

\begin{thm}\label{thm:hexagons}
Let $H$ be a (possibly degenerate) $o$-symmetric hexagon. If $H$ is weakly invariant, then $H$ is either a square, or similar to $\frac{1}{2}(T_0-T_0)$.
\end{thm}

To prove Theorem~\ref{thm:hexagons}, we need some lemmas. Before stating the first, for any $K \in \K$, let $\Sym_{aff}(K)$ denote the affine symmetry group of $K$, and observe that for any $K \in \M$, $\{I, -I \} \leq \Sym_{aff}(K)$, where $I$ is the identity.
For simplicity, if $H$ is an $o$-symmetric hexagon, we call the diagonals of $H$, containing $o$, the \emph{longest diagonals} of $H$.

\begin{lem}\label{lem:affinesymmetry}
Let $H$ be an $o$-symmetric hexagon such that $\{I, -I \} \lneqq \Sym_{aff}(H)$. Then we have one of the following.
\begin{itemize}
\item $H$ is a parallelogram, or
\item $H$ is affinely regular, or
\item $H$ has exactly one longest diagonal that is parallel to a pair of sides of $H$.
\end{itemize}
\end{lem}

\begin{proof}[Proof of Lemma~\ref{lem:affinesymmetry}]
First, note that if $H$ is degenerate, then it is a parallelogram.
Furthermore, if $H$ is nondegenerate and it has two longest diagonals that are parallel to some sides of $H$, then the same holds for the third longest diagonal,
which then implies that $H$ is affinely regular.
Thus, we need only show that if $H$ is nondegenerate, it has at least one longest diagonal parallel to a pair of its sides.

Let the vertices of $H$ be denoted by $a_1,a_2,\ldots,a_6=-a_3$, in counterclockwise order.
Let $g \in \Sym_{aff}(H) \setminus \{ I, -I \}$.
Observe that by symmetry, $g(o)=o$, and thus, $g$ is a linear transformation, permuting the vertices of $H$.
Then, for either $g$ or $-g$, one of the following holds.
\begin{enumerate}
\item[(i)] $g(a_1)=a_2$ and $g(a_2)=a_3$, which yields that $g(a_3)=-a_1$;
\item[(ii)] $g(a_1) = a_3$ and $g(a_2)=-a_1$, which yields that $g(a_3)=-a_2$;
\item[(iii)] $g(a_1)=a_2$ and $g(a_2)=a_1$, which yields that $g(a_3)=-a_3$;
\item[(iv)] $g(a_1)=a_3$ and $g(a_2)=a_2$, which yields that $g(a_3)=a_1$;
\item[(v)] $g(a_1)=-a_1$ and $g(a_2)=a_3$, which yields that $g(a_3)=a_2$.
\end{enumerate}

Let $a_3 = \alpha_1 a_1 + \alpha_2 a_2$, where $\alpha_1, \alpha_2 \in \Re$.

First, we examine (i).
In this case the matrix of $g$ in the basis $\{ a_1, a_2 \}$ is $\left[ \begin{array}{cc} 0 & \alpha_1 \\ 1 & \alpha_2 \end{array} \right]$.
Since $g(a_3)=-a_1$, it follows that
\[
\left[ \begin{array}{cc} 0 & \alpha_1 \\ 1 & \alpha_2 \end{array} \right] \left[ \begin{array}{c} \alpha_1 \\ \alpha_2 \end{array} \right]
= \left[ \begin{array}{c} -1 \\ 0 \end{array} \right] ,
\]
which implies that $\alpha_1 \alpha_2 = -1$ and $\alpha_1 + \alpha_2^2 = 0$. From this, we have $\alpha_2 = 1$ and $\alpha_1 = -1$.
Clearly, as $a_3 = a_2-a_1$, we obtained that $H$ is an affinely regular hexagon.
We can use the same idea in (ii), obtaining, again, that in this case $H$ is affinely regular.

Consider, now the case (iii). Then the matrix of $g$ in the basis $\{ a_1, a_2 \}$ is $\left[ \begin{array}{cc} 0 & 1 \\ 1 & 0 \end{array} \right]$, and we have
\[
\left[ \begin{array}{cc} 0 & 1 \\ 1 & 0 \end{array} \right] \left[ \begin{array}{c} \alpha_1 \\ \alpha_2 \end{array} \right]
= \left[ \begin{array}{c} - \alpha_1 \\ -\alpha_2 \end{array} \right] ,
\]
from which we obtain that $\alpha_2 = - \alpha_1$, and $a_3 = \alpha_1 \left( a_1 - a_2 \right)$, which yields that the diagonal $[a_3,-a_3]$ is parallel to the side $[a_1,a_2]$. To prove the assertion in the remaining two cases, we can apply the same idea.
\end{proof}

In the following, a rectangle $R$, circumscribed about $P$ and with sides parallel to $u, u^\perp \in \Sph^1$, is called \emph{locally extremal}, if there is some $\varepsilon > 0$ such that for any circumscribed rectangle $R'$ with sides parallel to $u', u'^\perp \in \Sph^1$, $\langle u, u' \rangle > 1-\varepsilon$ yields that $\frac{a(R)}{b(R)} \leq \frac{a(R')}{b(R')}$.

\begin{lem}\label{lem:parallelogram}
Let $P$ be a parallelogram, and let $R$ be a locally extremal rectangle circumscribed about $P$. If $R$ does not contain any side of $P$, then $P$ is a rhombus.
\end{lem}

\begin{proof}
Let the vertices of $R$ be $q_1$, $q_2$, $q_3$ and $q_4$, in counterclockwise order, and the vertices of $P$ be $p_1, p_2, p_3, p_4$ such that $p_i \in [q_{i-1},q_i]$ for $i=1,2,3,4$.

First, consider the case that a vertex of $P$, say $p_1$, coincides with a vertex of $R$, say $q_1$. Then $p_3=q_3$, and $p_2,p_4$ are interior points of $R$,
which clearly yields the assertion.

In the following we assume that for $i=1,2,3,4$, $p_i$ is a relative interior point of $[q_{i-1},q_i]$ (cf. Figure~\ref{fig:parallelogram}).

\begin{figure}[ht]
\includegraphics[width=0.55\textwidth]{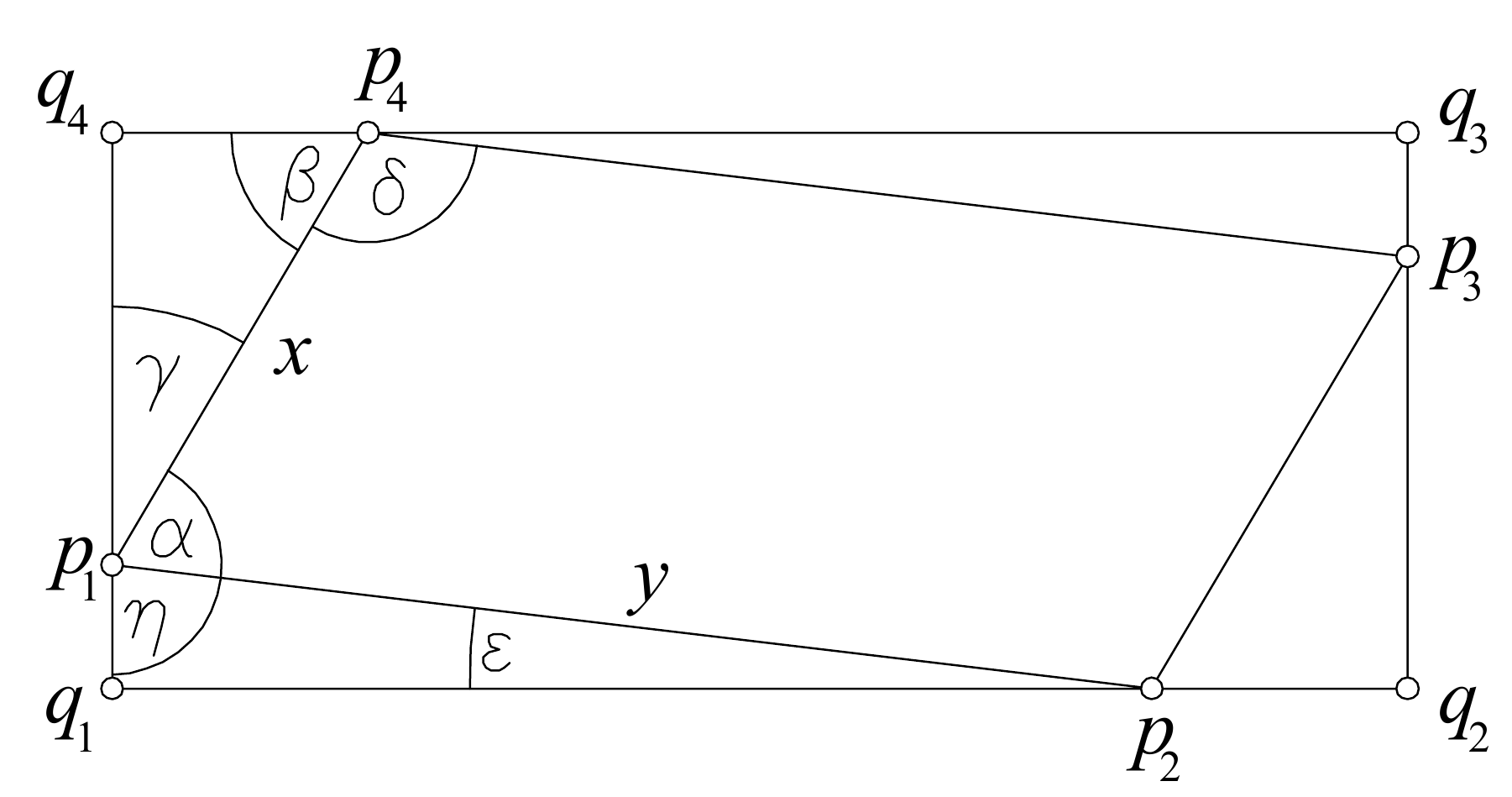}
\caption[]{An illustration for the proof of Lemma~\ref{lem:parallelogram}}
\label{fig:parallelogram}
\end{figure}

Let $x=|p_1-p_4| \leq y=|p_2-p_1|$, $\alpha = (p_4,p_1,p_2)\angle \leq \frac{\pi}{2}$, $\beta = (p_1,p_4,q_4)\angle < \alpha$, $\gamma = (p_4,p_1,q_4)\angle = \frac{\pi}{2} - \beta$,
$\delta = (p_1,p_4,p_3)\angle = \pi - \alpha$, $\varepsilon = (q_1,p_2,q_2)\angle = \alpha - \beta$ and $\eta = (q_1,p_1,p_2)\angle = \frac{\pi}{2} + \beta - \alpha$.

Assume that $|q_1-q_4| < |q_2-q_1|$.
Then
\begin{equation}\label{eq:ratio}
\frac{b(R)}{a(R)}= \frac{x \cos \beta + y \cos (\alpha-\beta)}{x \sin \beta + y \sin (\alpha-\beta)} .
\end{equation}
Regarding the right-hand side of \eqref{eq:ratio} as a function $h$ of $\beta$ and fixing $P$, the local extremality of $R$ implies that
$\frac{\dif h}{\dif \beta} = 0$.
Since
\[
\frac{\dif h}{\dif \beta} = \frac{x^2-y^2}{\left( x \sin \beta + y \sin(\alpha-\beta) \right)^2},
\]
it holds if, and only if $x=y$. This yields the assertion in this case.
If $|q_1-q_4| > |q_2-q_1|$ or $|q_1-q_4| = |q_2-q_1|$, we may apply a similar argument.
\end{proof}

\begin{proof}[Proof of Theorem~\ref{thm:hexagons}]
 We distinguish between two cases.

\emph{Case 1}, $H$ is a degenerate hexagon. Then $H$ is a parallelogram. Let the lengths of its sides be $x \leq y$, and its angles $\alpha \leq \frac{\pi}{2} \leq \beta = \pi - \alpha$.

First, assume that $H$ is not a rhombus; that is, that $x < y$. Then, by Lemma~\ref{lem:parallelogram}, any extremal rectangle $R$ contains a side of $H$.
A simple computation shows that this side must be of length $y$. If $H$ is a rectangle, then, clearly, it is not weakly invariant. On the other hand, if
$H$ is not a rectangle, i.e. $\alpha < \frac{\pi}{2}$, then $f_{H,R}(H)$ is a parallelogram with angles strictly greater than $\alpha$ and strictly less than $\beta$, and hence, it is not similar to $H$.

Now we consider the case that $H$ is a rhombus, and is not a square. Let $f_{H,R}$ be the affine transformation with the $x$-axis in a suitable Cartesian coordinate system as its axis, and ratio $0 < \lambda \neq 1$. Let $p_1=(x,y)$, with $x, y \geq 0$ be a vertex of $H$ in this coordinate system. Then, since the diagonals of a rhombus bisect each other and are perpendicular, the two adjacent vertices of $H$ are of the form $p_2=(-ty,tx)$ and $p_4=(ty,-tx)$
for some $t > 0$. Then the corresponding vertices of $f_{H,R}(H)$ are $p_1'=(x, \lambda y)$, $p_2'=(-ty, \lambda t x)$ and $p_4'=(ty,-\lambda tx )$.
Assume that $f_{H,R}(H)$ is similar to $H$. Then $f_{H,R}(H)$ is a rhombus, and thus, $|p_1'-p_2'|=|p_1'-p_4'|$. Since $t \neq 0$ and $\lambda \neq \pm 1$, a simple computation yields that in this case we have $a=0$ or $b=0$, implying that the $x$-axis is a symmetry axis of $H$.
As $f_{H,R}(R)$ is a square, in this case $f_{H,R}(H)$ is a square as well, and thus, it is not similar to $H$.

\emph{Case 2}, $H$ is a nondegenerate hexagon. Let the vertices of $H$ be $a_1$, $a_2$, $a_3$, $-a_1$, $-a_2$, $-a_3$ in counterclockwise order.
We prove by contradiction, and assume that $H$ is weakly invariant and is not similar to $\frac{1}{2}(T_0-T_0)$. Then Theorem~\ref{thm:triangles1} implies that $H$ is not affinely regular.

Let $R$ be an extremal rectangle of $H$ such that $f_{H,R}(H)$ is similar to $H$. By Theorem~\ref{thm:characterization}, it follows that $H$ is not strongly invariant, and thus, $R$ is not a square.
Let $g$ be the similarity such that $g(f_{H,R}(H)) = H$.
Then $f_{H,R} \circ g$ is an affine transformation belonging to $\Sym_{aff}(H)$.
Clearly, if $f_{H,R} \circ g \notin \{ I, -I \}$, and thus $\{ I, -I \} < \Sym_{aff}(H)$,  then, applying Lemma~\ref{lem:affinesymmetry}, we obtain that exactly one longest diagonal of $H$ is parallel to some side of $H$. Without loss of generality, we may assume that $[a_1,-a_1]$ is parallel to $[a_2,a_3]$.

%Let $g$ be the similarity such that $g(f_{H,R}(H))=H$.
Note that $f_{H,R} \circ g$ acts as a permutation on the vertices of $H$.
Since affine copies of parallel lines are parallel lines, and there is a unique diagonal of $H$ parallel to some sides, it follows that $g(f_{H,R}( \{ \pm a_1 \}))= \{ \pm a_1\}$. Thus, without loss of generality, we may assume that $g(f_{H,R}(a_2))=a_3$ and $g(f_{H,R}(a_3))=a_2$, implying that $g(f_{H,R}(a_1))=-a_1$.

\emph{Subcase 2.1}, $\bd R$ contains a pair of sides of $H$.

First, consider the case that $[a_2,a_3]$ is such a side. Then the axis of $f_{H,R}$ is either parallel to $[a_1,-a_1]$ or contains it.
Since $f_{H,R}(H)$ is similar to $H$, with the correspondence between the vertices of $H$ and $f_{H,R}(H)$ as described before Subcase 2.1,
either case yields that $H$ is axially symmetric to the line containing $[-a_1,a_1]$. Nevertheless, as $f_{H,R}(R)$ is a square, and $g(f_{H,R}(R))=R$, it follows that $R$ is a square, that is, $H$ is strongly extremal, which contradicts Theorem~\ref{thm:characterization}.

Now consider the case that $\bd R$ contains $[a_1,a_2]$ or $[a_3,-a_1]$, say, $[a_1,a_2]$. By the the similarity of $f_{H,R}$ and $H$ and the correspondence between the vertices of $H$ and $f_{H,R}(H)$ as described before Subcase 2.1, we obtain that the circumscribed rectangle $R'$ containing $[-a_1,a_3]$ is a square.
Let the vertices of $R$ be $q_1,q_2,q_3, q_4$ and the vertices of $R'$ be $p_1, p_2, p_3, p_4$ in counterclockwise order, respectively, such that $[a_1,a_2] \subset [q_1,q_2]$ and $[-a_3,a_1] \subset [p_4,p_1]$.
Let $\alpha=\angle(a_2,a_1,-a_1)$ and $\beta = \angle(-a_3,a_1,-a_1)$.

\begin{figure}[ht]
\includegraphics[width=0.4\textwidth]{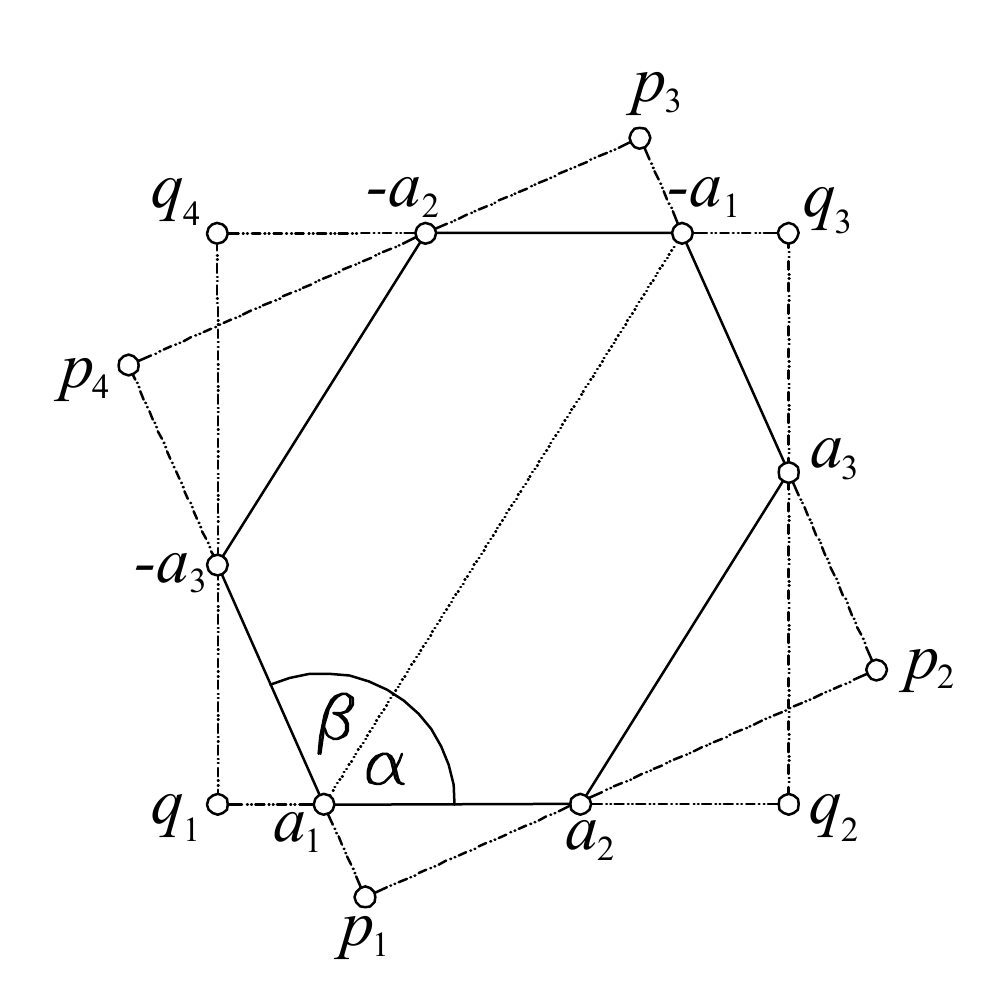}
\caption[]{An illustration for Subcase 2.1 of the proof of Theorem~\ref{thm:hexagons}}
\label{fig:hexagon}
\end{figure}

We distinguish the following four cases:
\begin{enumerate}
\item $\alpha+\beta \geq \frac{\pi}{2}$, and $\alpha, \beta < \frac{\pi}{2}$;
\item $\alpha + \beta < \frac{\pi}{2}$;
\item $\alpha < \frac{\pi}{2} \leq \beta$;
\item $\beta < \frac{\pi}{2} \leq \alpha$.
\end{enumerate}

First, consider the case that $\alpha+\beta \geq \frac{\pi}{2}$, and $\alpha, \beta < \frac{\pi}{2}$ (cf. Figure~\ref{fig:hexagon}).
Then the triangle $\conv \{ f_{H,R}(q_2), f_{H,R}(a_1), f_{H,R}(-a_3) \}$ is similar to the triangle $\conv \{ a_1,a_2,p_1\}$
such that $f_{H,R}(q_1)$, $f_{H,R}(a_1)$ and $f_{H,R}(-a_3)$ correspond to $p_1$, $a_1$ and $a_2$, respectively.
Observe that $\angle(p_1,a_1,a_2)=\angle(q_1,a_1,-a_3) = \pi - \alpha - \beta$.
On the other hand, since $f_{H,R}$ is not the identity, we have that $(f_{H,R}(p_1),f_{H,R}(a_1),f_{H,R}(a_2)) \angle \neq \pi - \alpha - \beta$, which is a contradiction. In the three remaining cases a similar argument can be applied.

\emph{Subcase 2.2}, $\bd R$ contains no pair of sides of $H$. Then, by Lemma~\ref{lem:parallelogram}, $\bd H$ contains exactly four vertices of $H$, which are the vertices of a rhombus. Since $H$ is not strongly invariant, it follows that $R$ and this rhombus are not squares.

First, assume that $H \cap \bd R = \{ \pm a_2, \pm a_3 \}$. Let $Q = \conv \{ f(\pm a_2), f(\pm a_3) \}$. Then $f_{H,R}(Q)$ is a rhombus similar to $Q$.
In the proof of Case 1 we have seen that the fact that $f_{H,R}(Q)$ is a rhombus yields that the axis of $f_{H,R}$
contains either $a_2$ or $a_3$, from which it follows that $\pm a_2$ and $\pm a_3$ are the midpoints of the sides of $R$.
Since $f_{H,R}(R)$ is a square, we obtain that $f_{H,R}(Q)$ is a square, and hence, $Q$ is a square and $f_{H,R} =I$, a contradiction.

Then we have that $\pm a_1$ belong to $\bd R$. Without loss of generality, we may assume that $\pm a_2 \in \bd R$.
Since $f_{H,R}(H)$ is similar to $H$, with the correspondence of their vertices as described before Subcase 2.1, it follows that there is a square
circumscribing $H$ and containing $\pm a_1$ and $\pm a_3$ in its boundary. Let $R'$ be this square. Without loss of generality, we may assume that the
vertices of  $R'$ are $p_1=(1,1)$, $p_2=(-1,1)$, $p_3=(-1,-1)$ and $p_4=(1,-1)$ in a suitable Cartesian coordinate system, and let $a_1 \in [p_3,p_4]$
and $a_3 \in [p_4,p_1]$ (cf. Figure~\ref{fig:rhombus}).
Let $u=(1,0)$ and $v=(0,1)$. Again without loss of generality, we may assume that the $x$-coordinate of $a_1$ is not positive.

\begin{figure}[ht]
\includegraphics[width=0.4\textwidth]{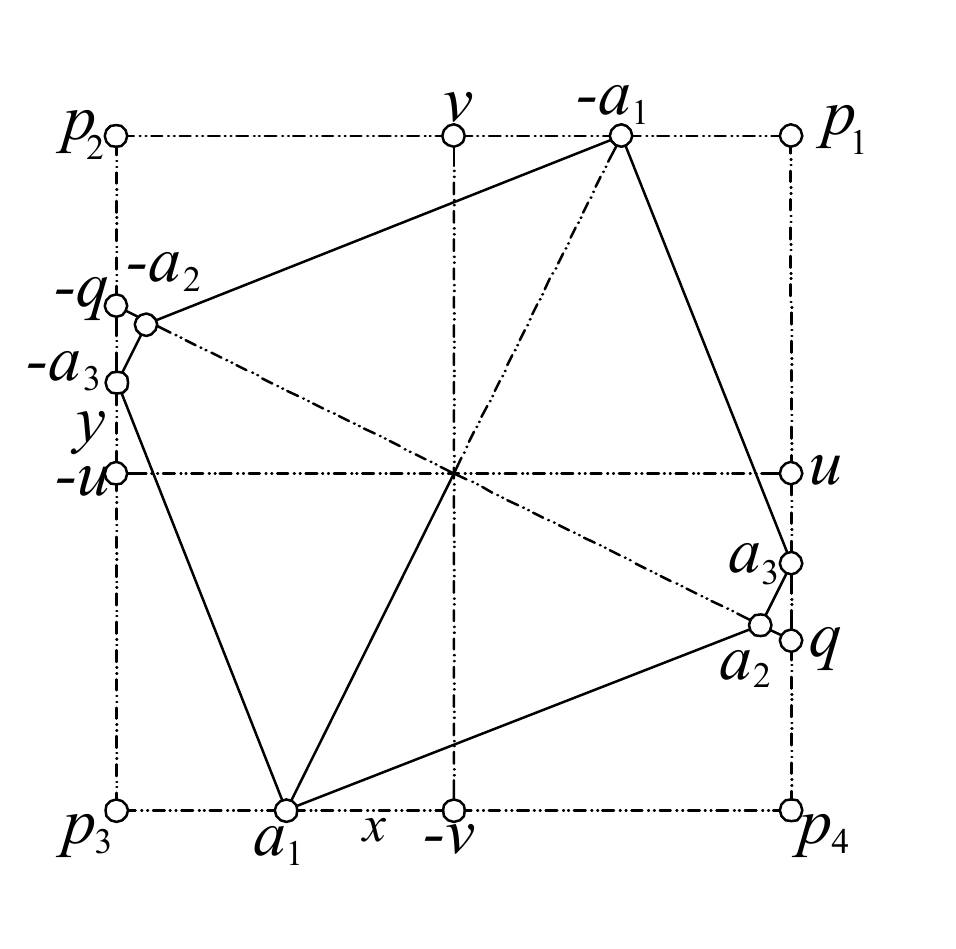}
\caption[]{An illustration for Subcase 2.2 of the proof of Theorem~\ref{thm:hexagons}}
\label{fig:rhombus}
\end{figure}

Let $q$ denote the intersection of the bisector of the segment $[a_1,-a_1]$ with $[u,p_4]$.
Note that as $\pm a_2$ and $\pm a_1$ are vertices of a rhombus, $\pm a_2 \in [-q,q]$.
%We note that the fact that $[a_2,a_3]$ is parallel to $[o,a_1]$, and that $\pm a_1$ and $\pm a_2$ are vertices of a rhombus, i.e. $a_2$ is on the bisector of the segment $[a_1,-a_1]$, the positions of $\pm a_1$ and $\pm a_3$ in Figure~\ref{fig:rhombus} determines the positions of $\pm a_2$.
Set $x=|a_1+v|=|u-q|$ and $y=|u-a_3|$.
Since $a_2$ is the intersection of $[-q,q]$ with the line through $a_3$, parallel to $[-a_1,a_1]$, we have $0 \geq y < x$.
Furthermore, as $f_{H,R}(H)$ is similar to $H$, there is some orthogonal affinity $f'$, with the $x$-axis as its axis and ratio $0 < \lambda \neq 1$, such that $f'(H)$ is similar to $H$. Thus, $f'(a_2)$ lies on the bisector of the segment $[-f'(a_1),f'(a_1)]$, and hence, the triangles $\conv \{ f'(a_1), -f'(v), o \}$ and $\conv \{ f'(a_2), u,o \}$ are similar, yielding $\frac{x}{\lambda} = \frac{\lambda y}{1}$; or in other words, $\lambda = \sqrt{\frac{x}{y}} > 1$.

Now we use the fact that $f'(H)$ is similar to $H$, where $f'(a_1)$, $f'(a_2)$ and $f'(a_3)$ correspond to $-a_1$, $a_3$ and $a_2$, respectively.
An elementary computation yields that
\[
|a_1| = \sqrt{1+x^2}, \quad |a_2| = \frac{1+xy}{\sqrt{1+x^2}}, \quad |a_3| = \sqrt{1+y^2} ,
\]
\[
|f'(a_1)|=\frac{x}{y} \sqrt{1+y^2}, \quad |f'(a_2)| = \frac{(1+xy)x \sqrt{1+y^2}}{(1+x^2)y}, \quad |f'(a_3)| = \sqrt{1+x^2}.
\]
Substituting these expressions into the equalities $\frac{|f'(a_1)|}{|a_1|} = \frac{|f'(a_2)|}{|a_3|} = \frac{|f'(a_3)|}{|a_2|}$,
and solving the obtained system of equations, it follows that $x=y$, which contradicts our observation that $y < x$.
\end{proof}

\section{Remarks and questions}\label{sec:remarks}

\begin{rem}
If $C_1, C_2, \ldots$ is a sequence of plane convex bodies such that $\SIM(C_{k+1}) \in F(C_k)$ for $k=1,2,\ldots$, and $\lim\limits_{k \to \infty} C_k = C$ exists in the topology defined by Hausdorff metric, then $C$ is invariant in $\AFF(C_1)$.
\end{rem}

\begin{proof}
Note that $C_k \in \AFF(C_1)$ for every value of $k$.
Furthermore, if $\{ C_k \}$ is a convergent sequence of plane convex bodies, and $R_k$ is any extremal rectangle of $C_k$, then every accumulation point of the sequence $\{ R_k \}$ is an extremal rectangle of $C$. Thus, for any $\SIM(C_k') \in F(C_k)$, we have $\lim\limits_{k \to \infty}  C'_k = \lim\limits_{k \to \infty}  C'_{k+1} \in SIM \left( \lim\limits_{k \to \infty} C_k \right)$. Now, if we choose $C_k'=C_k$, the assertion immediately follows.
\end{proof}

\begin{rem}
Among $o$-symmetric hexagons, only the affine class of the regular hexagon contains any (weakly) invariant element. Thus, for any sequence of $o$-symmetric hexagons $\{ H_k\}$ not in this class and satisfying $\SIM(H_{k+1}) \in F(H_k)$, $\{H_k \}$ is not convergent.
\end{rem}

Convexity does not play an important role in the results presented in this paper; we utilized convexity when we resolved the non-uniqueness of the circumscribed rectangle. Otherwise, most results remain valid for non-convex objects.
In particular, the dimension $d$ of the phase space of the discrete dynamics defined by $f_{K,R}$ is always $d=2$, regardless of the convexity of the investigated iterated shape.
In the case of the triangle we found that there is one, global, codimension 2 (point) attractor in the phase space. Numerical computations suggest that the attractors are always periodic orbits.
One example of our computations is shown in  Figure \ref{fig:Mark_trafok} where we scanned the phase space of a `randomly chosen' convex heptagon.  The affine copies of the heptagon are characterized, up to similarity, by two angles $\alpha$ and $\beta$ of the unique maximum area triangle inscribed in it and we use these variables as the coordinates of the phase space. We note that this triangle is preserved under affine transformations of the heptagon. We started iterations on a uniform
($\alpha,\beta$) grid in the phase space and iterated the gridpoints 60 times. Subsequently we omitted the first 10 iterates and plotted the orbits. Even though the initial values  are distributed uniformly over the $(\alpha,\beta)$ plane, all their transformed copies seem to accumulate on a 5-cycle. We certainly did \emph{not} find in any our computations any planar shape the discrete dynamics of which appeared to have either quasi-periodic or ergodic components. Our question is related to this fact:

\begin{ques} Is it true that, regardless of the initial shape, the discrete dynamics defined by $f_{K,R}$ always has periodic attractor(s)? \end{ques}

\begin{figure}[ht]
\begin{center}
\includegraphics[width=32 mm]{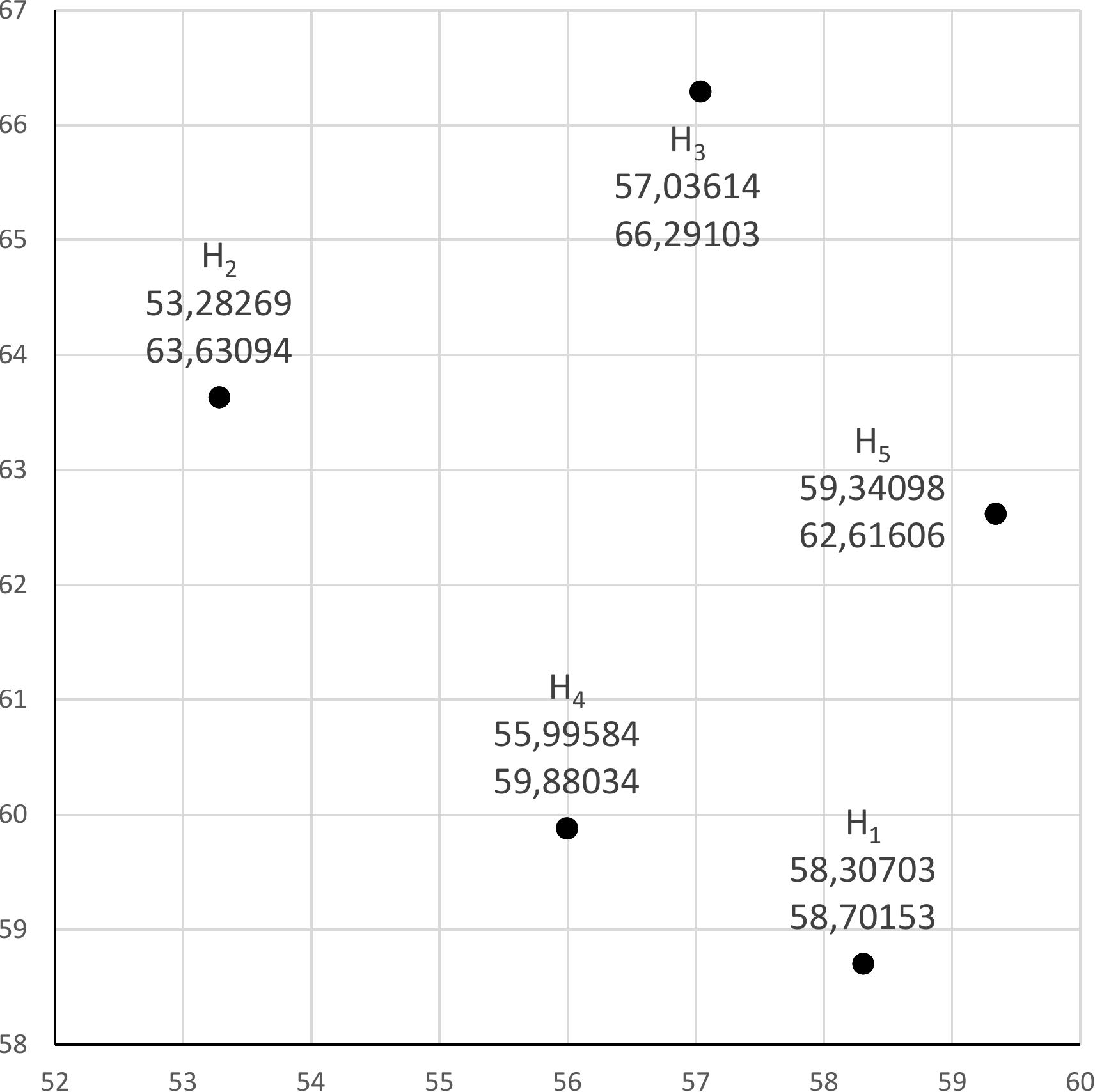}
\includegraphics[width=80mm]{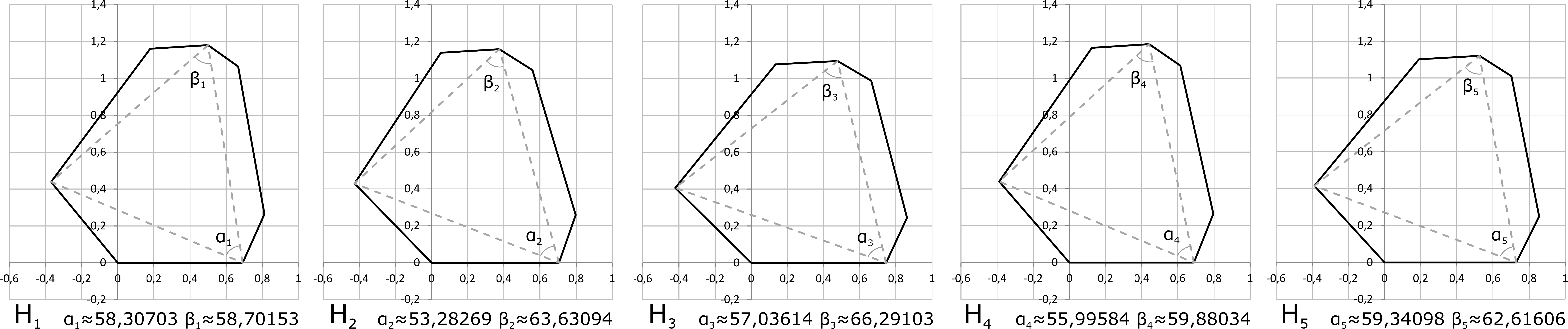}
\end{center}
\caption[]{Left panel: Globally attractive 5-cycle in the phase space of a heptagon. The phase space was scanned by starting orbits of length 60 at the meshpoints of a global orthogonal grid with 5 degree spacing (17x35=595 meshpoints). The figure was obtained by omitting the first 10 elements of each orbit.
Right panel: physical shapes.}
\label{fig:Mark_trafok}
\end{figure}

It is possible to extend our definition to higher dimensions in more than one way. For a convex body in $\Re^n$, we may call a rectangular box $R$ circumscribed about $K$ with side lengths $a_1(R) \leq a_2(R) \leq \ldots \leq a_n(R)$ \emph{extremal}, if $\frac{a_1(R)}{a_n(R)}$ is minimal over the family of all rectangular boxes circumscribed about $K$.
Then $f_{K,R}$ may be defined as one of the following:
\begin{itemize}
\item $f_{K,R}$ is the orthogonal affinity, with its hyperplane perpendicular to the sides of length $a_n(R)$, and ratio $\frac{a_1(R)}{a_n(R)}$. The image of $R$ under this transformation is a rectangular box with side lengths $a_1(R) = a_1(R) \leq a_2(R) \leq \ldots \leq a_{n-1}(R)$.
\item $f_{K,R}$ is the composition of $n-1$ orthogonal affinities, where the hyperplane of the $i$th affinity is perpendicular to the sides of length $a_{n+1-i}(R)$, and its ratio is $\frac{a_1(R)}{a_{n+1-i}(R)}$. The image of $R$ under this transformation is a cube of side length $a_1(R)$.
\end{itemize}
In the case of $n$ spatial dimensions it is easy to show that the dimension $d$ of the phase space will be given by $d=(n^2+n-2)/2$.
While from the point of view of shape evolution processes it seems interesting to find higher dimensional analogues of our model, the investigation of the above two transformations is beyond the scope of this paper.

\section{Acknowledgements}
The authors gratefully acknowledge the support of the J\'anos Bolyai Research Scholarship of the Hungarian Academy of Sciences and support from OTKA grant 119245.

%....................


\begin{thebibliography}{99}
\bibitem{Blott} S.J. Blott and K. Pye, \emph{Particle shape: A review and new methods of characterization and classification}, Sedimentology \textbf{55} (2008), 31-63.
\bibitem{Bloore} F. Bloore, \emph{The shape of pebbles}, J. Int. Ass. Math. Geol. \textbf{9} (1977), 113–122.
\bibitem{Bluck} B.J. Bluck, \emph{Sedimentation of beach gravels; examples of South Wales}, J. Sediment. Res. \textbf{37} (1967), 128-156.
\bibitem{BMP05} P. Brass, W. Moser and J. Pach, \emph{Research problems in discrete geometry}, Springer, New York, 2005.
\bibitem{Carr} A.P. Carr, \emph{Size grading along a pebble beach: Chesil beach, England}, J. Sediment. Petrol. \textbf{39} (1969), 297-311.
\bibitem{Cox} E.P. Cox, \emph{A method of assigning numerical and percentage values to the degree of roundness of sand grains}, J. Paleontol. \textbf{1} (1927), 179-183.
\bibitem{Devaney} R.L. Devaney, \emph{An introduction to chaotic dynamical systems}, second edition, Westview Press, New York NY, USA, 2003.
\bibitem{Domokosetal2} G. Domokos, D.J. Jerolmack, A.\'A. Sipos and \'A. T\"or\"ok, \emph{How River Rocks Round: Resolving the Shape-Size Paradox}, PLoS ONE \textbf{9}(2) (2014), e88657. doi:10.1371/journal.pone.0088657
\bibitem{DSV} G. Domokos, A.\'A. Sipos and P.L. V\'arkonyi, \emph{Continuous and discrete models for abrasion processes}, Period. Polytech. Architecture  \textbf{40}(1) (2009), 3-8.
\bibitem{DG1} G. Domokos and G.W. Gibbons, \emph{The evolution of pebble size and shape in space and time}, Proc. R. Soc. Lond. A \textbf{468} (2012), 3059-3079.
\bibitem{DG2} G. Domokos and G.W. Gibbons, \emph{Geometrical and physical models of abrasion}, (2013), arXiv preprint arXiv:1307.5633.
\bibitem{Firey} W.J. Firey, \emph{The shape of worn stones}, Mathematika \textbf{21} (1974), 1-11.
\bibitem{Krapivsky} P.L. Krapivsky and S. Redner, \emph{Smoothing rock by chipping}, Phys. Rev E. \textbf{75} (2007), 031119.
\bibitem{G07} P. M. Gruber, \emph{Convex and discrete geometry}, Springer-Verlag, New York, 2007.
\bibitem{L98} K. Leichtweiss, \emph{Affine geometry of convex bodies}, Barth, Heidelberg, Germany, 1998.
\bibitem{Sch93} R. Schneider, \emph{Convex bodies: The Brunn-Minkowski theory}, Encyclopedia of Mathematics and its Applications \textbf{44}, Cambridge University Press, Cambridge, 1993.
\bibitem{Zingg} T. Zingg, \emph{Beitrag zur Schotteranalyse}, Mineralogische und Petrologische Mitteilungen \textbf{15} (1935), 39-140.
\end{thebibliography}
\end{document}